\documentclass{article}
\title{A tale of two balloons}
\author{Omer Angel
  \and Gourab Ray
  \and Yinon Spinka}
\date{March 2021}
\usepackage{amsmath,amssymb,amsfonts,amsthm}

\usepackage[colorlinks=true,citecolor=black,urlcolor=blue,pdfborder={0 0 0}]{hyperref}

\hypersetup{linkcolor=[rgb]{0,0,0.6}}

\usepackage{graphicx}
\usepackage{enumitem}
\usepackage{color}

\usepackage{tikz}

\usepackage[margin=1cm]{caption}
\usepackage[protrusion=true
]{microtype}

\usepackage[nameinlink]{cleveref}
  \crefname{theorem}{Theorem}{Theorems}
  \crefname{thm}{Theorem}{Theorems}
  \crefname{mainthm}{Theorem}{Theorems}
  \crefname{lemma}{Lemma}{Lemmas}
  \crefname{lem}{Lemma}{Lemmas}
  \crefname{remark}{Remark}{Remarks}
  \crefname{prop}{Proposition}{Propositions}
  \crefname{defn}{Definition}{Definitions}
  \crefname{corollary}{Corollary}{Corollaries}
  \crefname{section}{Section}{Sections}
  \crefname{figure}{Figure}{Figures}
  \crefname{problem}{Problem}{Problems}

\newtheorem{thm}{Theorem}

\newtheorem{claim}[thm]{Claim}
\newtheorem{lemma}[thm]{Lemma}

\newtheorem{prop}[thm]{Proposition}
\theoremstyle{definition}

\newtheorem{remark}[thm]{Remark}
\newtheorem*{remark*}{Remark}

\newtheorem{problem}[thm]{Problem}

\renewcommand{\P}{\mathbb P}
\newcommand{\Z}{\mathbb Z}
\newcommand{\E}{\mathbb E}
\newcommand{\R}{\mathbb R}
\newcommand{\N}{\mathbb N}
\newcommand{\HH}{\mathbb H}
\newcommand{\T}{\mathbb T}
\newcommand{\eps}{\varepsilon}

\newcommand{\cB}{\mathcal B}
\newcommand{\cE}{\mathcal E}
\newcommand{\cT}{\mathcal T}

\AtEndDocument{
  \bigskip
  \small
  \textsc{Omer Angel, Yinon Spinka} \par
  \textsc{Department of Mathematics, University of British Columbia} \par
  \textit{Email:} \texttt{\{angel,yinon\}@math.ubc.ca} \par
  \par
  \textsc{Gourab Ray} \par
  \textsc{Department of Mathematics, University of Victoria} \par
  \textit{Email:} \texttt{gourab@math.uvic.ca}
}

\usepackage{geometry}

\newcommand{\bn}{\mathbf n}
\newcommand{\bu}{\mathbf u}

\begin{document}

\maketitle
\begin{abstract}
  From each point of a Poisson point process start growing a balloon at rate 1.
  When two balloons touch, they pop and disappear.
  Is every point contained in balloons infinitely often or not?
  We answer this for the Euclidean space, the hyperbolic plane and regular trees.
  
  The result for the Euclidean space relies on a novel $0$-$1$ law for stationary processes.
  Towards establishing the results for the hyperbolic plane and regular trees, we prove an upper bound on the density of any well-separated set in a regular tree which is a factor of an i.i.d.\ process. 
\end{abstract}

\section{Introduction}

We consider the following process on a metric space $(\Omega,\rho)$.
Initially, at time $0$, we are given a point process $\Pi$ in $\Omega$.
The \textbf{balloon process} on $\Omega$ started from $\Pi$ is defined as follows.
Start growing a ball, or \textbf{balloon}, around each point $x$ in the support $[\Pi]$, so that at time $t$ the balloons are balls of radius $t$.
At any time that two of the balloons touch each other, they both immediately pop and are removed.
We emphasize that there is no additional randomness in the process beyond the initial point process $\Pi$.
It is not a priori clear that this process is well defined, and indeed this requires $\Pi$ to satisfy some simple conditions; for cases of interest, these are indeed satisfied.

While the above can be fairly general, our main interest is when $(\Omega,\rho,\mu)$ is a metric measure space and $\Pi$ is the Poisson process on $\Omega$ with intensity measure $\mu$. In this case, we simply call the balloon process the \textbf{Poisson balloon process} on $\Omega$. See \cref{sec:basic} for some basic properties of this process, including the fact that it is well defined.


\begin{figure}[h]
 \centering
 \includegraphics[scale=0.25]{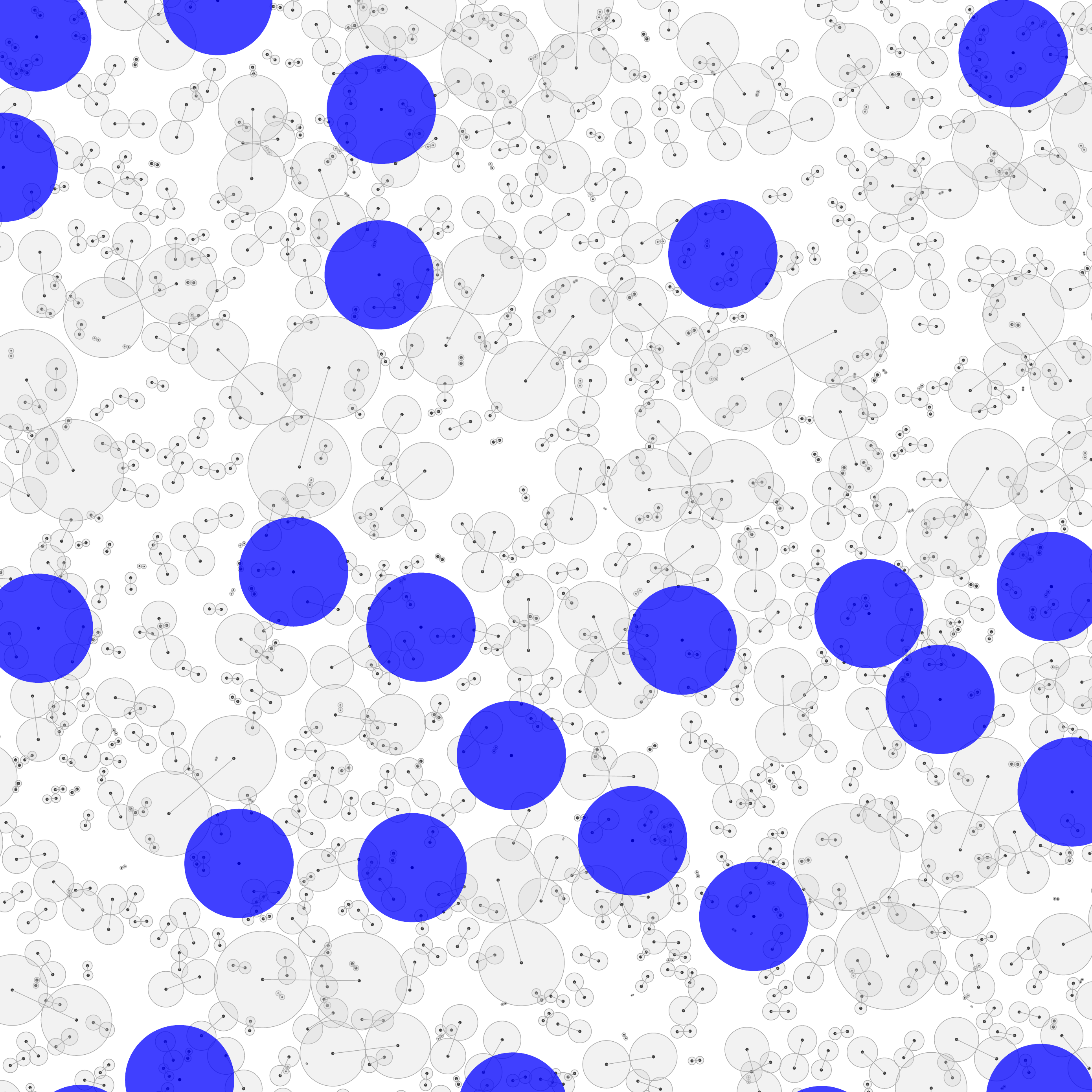}\qquad
 \includegraphics[scale=0.25]{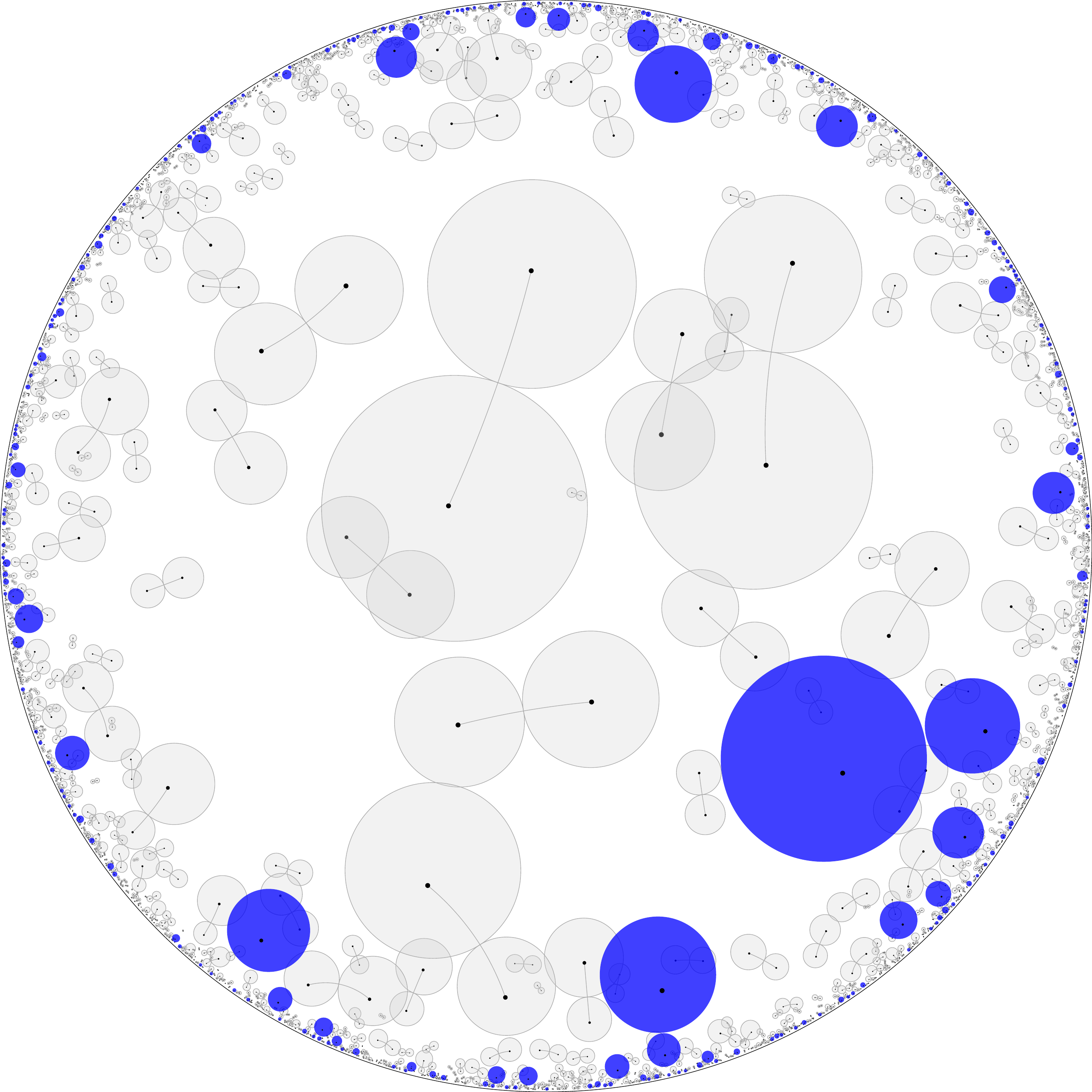}
 \caption{Simulations of the Poisson balloon process in the Euclidean plane (left) and the hyperbolic plane (right), shown at a fixed time $t$.
   Balloons still active at time $t$ are shown in blue.
   Balloons which have popped before time $t$ are shown in gray at their size  when they popped.
   Edges connect the centers of balloons which popped each other.}
\end{figure}


Denote by $\Pi_t$ the point process of centers of balloons that are still active at time $t$.
The balloons at time $t$ form a disjoint collection $\{B(x,t) : x\in [\Pi_t]\}$, where $B(x,r)$ denotes the ball of radius $r$ around $x$.
The union of this collection is the set of points covered by a balloon at time $t$ and is denote by
\[ S_t := \{x\in\Omega : \rho(x,[\Pi_t]) \leq t\} ,\]
where $\rho(x,S)$ is the distance between a point $x$ and a set $S$.
We say that the balloon process is \textbf{recurrent} if, almost surely, for every $x\in\Omega$, the set of times $t$ at which $x\in S_t$ is unbounded.
We say that the balloon process is \textbf{transient} if, almost surely, for every $x\in\Omega$, the set of times $t$ at which $x\in S_t$ is bounded.
We say the balloon process is transient (or recurrent) at $x$ if these statements hold for a given $x\in\Omega$.
Note that with no further assumptions, it is possible for the process to be neither transient nor recurrent (if the type depends on $x$, or if the events have non-trivial probabilities).
Fix a point $o\in \Omega$ called the origin.
Denote by
\[ R_t := \rho(o,[\Pi_t]) \]
the distance to the center of the balloon nearest to the origin at time $t$.
Note that the origin belongs to $S_t$ if and only if $R_t \le t$.
In particular, if $\liminf R_t/t<1$ almost surely then the balloon process is recurrent, and if $\liminf R_t/t>1$ almost surely then it is transient.

Our main results give a surprising difference between the Euclidean space and the hyperbolic plane in this regard:

\begin{thm}\label{thm:main}
  For any $d \ge 1$, the Poisson balloon process on $\R^d$ is recurrent.
  Moreover,
  \[ \liminf_{t \to \infty} \frac{R_t}{t} = 0 \text{ almost surely.} \]
\end{thm}


\begin{thm}\label{thm:main2}
  The Poisson balloon process on the hyperbolic plane $\HH$ is transient.
  Moreover,
  \[ \liminf_{t \to \infty} \frac{R_t}{t}
    \ge \frac{\log 2}{\log\frac{1+\sqrt5}{2}} \approx 1.44 \text{ almost surely.} \]
\end{thm}

In \cref{thm:main}, the metric $\rho$ is the standard Euclidean metric and the measure $\mu$ is Lebesgue measure (the scaling, or equivalently the intensity of the Poisson point process, is immaterial for the question). In \cref{thm:main2}, the metric $\rho$ is the standard hyperbolic metric with any constant curvature and $\mu$ is the corresponding measure (here a scaling of the measure would be equivalent to a change of curvature). We mention that, unlike in the Euclidean space, in the hyperbolic plane it is not a priori clear that recurrence/transience do not depend on the curvature/intensity of the Poisson point process.


We also consider the balloon process on trees.
Let $\cT_d$ denote the infinite $d$-regular real tree with unit edge lengths,
that is, the $d$-regular tree in which each edge is a copy of the unit interval $[0,1]$ (endowed with the Lebesgue measure and standard metric),
and endpoints are identified at each vertex. This naturally defines a metric measure space.
Fix an arbitrary point $o$ to be the origin (not necessarily one of the vertices).

\begin{thm}\label{thm:main3}
  For any $d\geq 3$, the Poisson balloon process on $\cT_d$ is transient.
  Moreover, almost surely, $R_t \ge 2t - O(\log t)$ as $t \to \infty$, and
  \[ \liminf_{t \to \infty} \frac{R_t}{t} = 2  \text{ almost surely}. \]
\end{thm}

A general argument due to \cite{holroyd2009poisson} yields an upper bound of 2 for $\liminf_{t \to \infty} R_t/t$ for reasonable spaces, including those considered above (see \cref{lem:liminf-upper-bound}).
\cref{thm:main} shows that this quantity is in fact zero for the Euclidean space, and \cref{thm:main3} shows that it is exactly equal to 2 for regular trees.
It is however unclear if for the hyperbolic plane this quantity is actually 2 or not.

Towards proving \cref{thm:main}, we came across the following fact.
While it seems like a classic result, it appears not to be well known,
and as far as we can tell it is new for $d>1$.
In the one-dimensional case it has been proved by Tanny \cite{tanny1974zero}.
Their proof uses a connection to branching processes in random environment and  a criterion for their survival from \cite{tannythesis}.
Our proof has the advantages of being more direct and applying in any dimension.

\begin{thm}\label{T:limsup}
  Let $(X_\bn)_{\bn \in \Z^d}$ be a translation-invariant process.
  Then almost surely
  \[ \limsup_{\|\bn\| \to \infty} \frac{X_\bn}{\|\bn\|} \in \{0,+\infty\}. \]
\end{thm}

Note that the choice of norm is not important.
If the variables $X_{\bn}$ are i.i.d., then a simple application of Borel--Cantelli yields that the $\limsup$ is $\infty$ if and only if $\sum n^{d-1} \P(X \ge n) = \infty$,
or equivalently, $\E (X^+)^d = \infty$.
If instead of translation invariance we only assume that all $X_\bn$ have the same distribution, 
then it is not hard to construct examples where the $\limsup$ above takes non-trivial values.
We remark also that the same result holds if $X$ is some random field indexed by $\R^d$, or by the points of some translation-invariant point process on $\R^d$,
since one may apply the theorem to a new process $X'_\bn$, defined to be the maximum of $X$ on a unit box around $\bn\in\Z^d$.

Our proof of \cref{thm:main2} relies on approximating the hyperbolic plane by a $3$-regular tree and thus relies on the arguments used for regular trees. In turn, for the proof of \cref{thm:main3}, we rely on a new upper bound of the density of well-separated sets which are factors of i.i.d.\ processes on regular trees (see \cref{thm:tree-independent-set-bound}).
This is related to previous works on independence ratio due to Bollob\'{a}s \cite{bollobas1981independence}, McKay \cite{mckay1987}, and Rahman and Virag \cite{rahman2017local}.

\paragraph{Acknowledgement:}

We would like to thank Itai Benjamini for suggesting this problem.
We are grateful to Thomas Budzinski for an elegant proof of \cref{T:limsup}, and to Ofer Zeitouni for bringing \cite{tanny1974zero} to our attention.

\section{Balloon basics}\label{sec:basic}

In this section, we explore basic properties of the Poisson balloon process, including the fact that it is well defined.
The balloon process associated to a point process $\Pi$ on a metric space $(\Omega,\rho)$ is closely related to the so-called stable matching between the points of $[\Pi]$.
Those have been studied in depth in $\R^d$ in \cite{haggstrom1996nearest,holroyd2009poisson}.
Some of the relevant arguments extend beyond $\R^d$ to our settings.
We give a brief sketch of them here and refer the reader to those papers for details.

Consider a locally finite set of points $U \subset \Omega$.
A \textbf{partial matching} of $U$ is a collection of disjoint unordered pairs $m = \{\{u_i,v_i\}\}_{i\in I}$ of elements of $U$.
An element $u\in U$ is \textbf{matched} to $v \in U$ if $\{u,v\}$ is one of the pairs;
in this case we write $m(u):=v$.
A \textbf{matching} is a partial matching in which every vertex is matched.
A partial matching $m$ is \textbf{stable} if there do not exist distinct points $u,v \in U$ such that $\rho(u,v) < \min\{ \rho(u,m(u)),\rho(v,m(v))\}$, where $\rho(x,m(x))$ is defined to be $\infty$ when $x$ is unmatched.
The set $U$ is \textbf{non-equidistant} if there do not exist $u,v,u',v' \in U$ with $\{u,v\} \neq \{u',v'\}$ and $\rho(u,v)=\rho(u',v')$.
A \textbf{descending chain} in $U$ is an \emph{infinite} sequence $x_0,x_1,\dots \in U$ such that $\rho(x_n,x_{n+1})$ is strictly decreasing.

The following is shown in \cite[Lemma 15]{holroyd2009poisson} (the statement there is only in $\R^d$, but the arguments apply in general metric spaces).
If $U$ is discrete, non-equidistant and has no descending chains, then it has a unique stable partial matching.
Furthermore, this partial matching contains at most one unmatched point, and it can be produced by the \emph{greedy} or \emph{iterated mutually closest matching} algorithm:
Match all mutually closest pairs, remove these points and repeat indefinitely.

\begin{lemma}\label{lem:balloon-matching}
  Suppose that $U$ is discrete, non-equidistant and has no descending chains.
  Let $m$ be the unique stable partial matching of $U$. Define
  \[ U_t := \{x\in U : \rho(x,m(x)) > 2t\}, \]
  where as before $\rho(x,m(x))=\infty$ for unmatched points.
  Then $U_t$ is precisely the set of active balloon centers at time $t$ for the balloon process started from $U$.
\end{lemma}

\begin{proof}
  We need to check two things:
  that the balls of radius $t$ around points of $U_t$ are disjoint and that they are only removed when touching another balloon.
  Suppose $u,v\in U_t$ have intersecting balloons, so that $\rho(u,v)\le 2t$. However, $\rho(u,m(u))>2t$ since $u\in U_t$, and similarly for $\rho(v,m(v))$.
  Thus such $u,v$ violate the stability of $m$.

  Secondly, suppose $u\in U_s$ for all $s<t$ but $u\not\in U_t$.
  This implies that $\rho(u,m(u))=2t$, and so the balloon around $u$ touches the balloon around $m(u)$ at time $t$.
\end{proof}

The balloon process associated to a point process $\Pi$ is therefore well defined if, almost surely, $[\Pi]$ is discrete, non-equidistant and has no descending chains.
It is not hard to show that this is the case when $\Pi$ is a Poisson point process in a ``nice'' metric measure space $(\Omega,\rho,\mu)$.
These include all spaces we consider in this paper.
For future reference, we give some sufficient conditions.
It is straightforward that $[\Pi]$ is almost surely discrete whenever $\mu$ is non-atomic and locally finite and $(\Omega,\rho)$ is a separable metric space.
The other two properties are addressed in the following propositions.
For $x \in \Omega$ and $R \ge r \ge 0$, denote $A(x,r,R) := \{y \in \Omega :r\le \rho(x,y)\le R\}$.

\begin{prop}\label{prop:non_equidistant}
  Let $(\Omega, \rho, \mu)$ be a metric measure space satisfying
  \begin{itemize}[nosep]
  \item $(\Omega,\mu)$ is a $\sigma$-finite measure space.
  \item $(\Omega,\rho)$ is a separable metric space.
  \item Very thin annuli: For all $r \ge 0$ and $x \in \Omega$ we have $\mu(A(x,r,r)) = 0$.
  \end{itemize}
  Then the Poisson point process $\Pi$ on $\Omega$ with intensity measure $\mu$ is almost surely non-equidistant.
\end{prop}

\begin{proof}
Let $S \subset \Omega$ be a measurable set, and let $\cE(S)$ be the event that there exist $w,x,y,z \in [\Pi]$ with $w \in S$, $x,y,z \notin S$ and $\rho(w,x)=\rho(y,z)$. Let us show that $\cE(S)$ has probability zero. Indeed, $\cE(S)$ is precisely the event that $[\Pi]$ has a point in $S \cap A(x,r,r)$ for some $x \in [\Pi]\setminus S$ and $r \in \{ \rho(y,z) : y,z \in [\Pi]\setminus S\}$. Conditioning on $[\Pi] \setminus S$, the very thin annuli assumption implies that this has probability zero for any fixed $x$ and $r$. Since there are countably many relevant $x$ and $r$, the conditional (and hence also unconditional) probability of $\cE(S)$ is zero.

Now let $U \subset \Omega$ be a countable dense set, and consider the countable collection $\{S_n\}_n$ of all balls with rational radius and centers in $U$. Then, almost surely, no event $\cE(S_n)$ occurs. Since the existence of equidistant points is easily seen to imply the occurrence of some $\cE(S_n)$, we conclude that $[\Pi]$ is almost surely non-equidistant.
\end{proof}

\begin{prop}\label{prop:descending_chain}
  Let $(\Omega, \rho, \mu)$ be a metric measure space satisfying 
  \begin{itemize}[nosep]
  \item $(\Omega,\mu)$ is a $\sigma$-finite measure space.
  \item Uniformly thin annuli: For all $r\ge 0$ there exists a $\delta >0$ such that 
    \[ \sup_{x \in \Omega} \mu(A(x,r,r+\delta)) \le 1.\]
  \end{itemize}
  Then the Poisson point process $\Pi$ on $\Omega$ with intensity measure $\mu$ almost surely has no descending chains.
\end{prop}

\begin{proof}
  It suffices to show that there almost surely does not exist a sequence of distinct $x_0,x_1,\ldots \in [\Pi]$ such that $\rho(x_n,x_{n+1})$ converges.
  In turn, it suffices to show that for any $r \ge 0$ there exists $\delta>0$ such that there almost surely does not exist a sequence $x_0,x_1,\ldots \in [\Pi]$ such that $\rho(x_n,x_{n+1}) \in [r,r+\delta]$ for all $n$. By the thin annuli assumption, it suffices to show that the latter holds whenever $r$ and $\delta$ are such that $\mu(A(x,r,r+\delta)) \le 1$ for all $x \in \Omega$.

Fix a set $S \subset \Omega$ with $\mu(S)<\infty$ and consider the process $(U_n)_{n \ge 0}$ defined by $U_0 := [\Pi] \cap S$ and $U_{n+1} := [\Pi] \cap A(U_n,r,r+\delta) \setminus (U_0 \cup \cdots \cup U_n)$, where we write $A(U,r,R) := \bigcup_{x \in U} A(x,r,R)$ for a set $U$.
 Observe that if there exists a sequence $x_0,x_1,\ldots \in [\Pi]$ such that $x_0 \in S$ and $\rho(x_n,x_{n+1}) \in [r,r+\delta]$ for all $n$, then there also exists a sequence $y_0,y_1,\ldots$ such that $y_n \in U_n$ for all $n$. However, $(|U_n|)_{n \ge 0}$ is stochastically dominated by the progeny sequence of a branching process whose offspring distribution is Poisson with mean $1$, which almost surely goes extinct, so that the latter occurs with probability zero. Since $\mu$ is $\sigma$-finite, this shows that $[\Pi]$ almost surely has no descending chains.
\end{proof}

We end this section with a property of the Poisson balloon process which follows from the methods of~\cite{holroyd2009poisson},
where (for $\R^d$) it is basically the content of equation (19) there.

\begin{lemma}\label{lem:liminf-upper-bound}
Let $(\Omega, \rho, \mu)$ be a metric measure space satisfying the assumptions of \cref{prop:non_equidistant,prop:descending_chain} and having $\mu(B(x,r)) <\infty$ for all $x \in \Omega$ and $r>0$. Further suppose that the Poisson point process $\Pi$ with intensity measure $\mu$ on $\Omega$ is ergodic under the action of a group of measure-preserving isometries of $\Omega$. Then the Poisson balloon process satisfies 
  \begin{equation}
    \liminf_{t \to \infty} \frac{R_t}t \le 2 \qquad\text{almost surely.} \label{eq:upper_limsup}
  \end{equation}
\end{lemma}

\begin{proof}[Sketch of proof]
  Recall that \cite[eq.\ (19)]{holroyd2009poisson} states that for the stable matching $m$ of $[\Pi]$ in $\R^d$, there are almost surely infinitely many points $x \in [\Pi]$ such that $\rho(x,m(x)) \ge \rho(x,o)-1$, where $o$ is the origin.
 Actually, the same proof goes through when $\rho(x,o)-1$ is replaced by $\rho(x,o)-r$, where $r>0$ is arbitrary.
 We claim that this extends to our general setting as long as $o \in \Omega$ and $r>0$ are chosen so that $\mu(B(o,r))>0$. The argument leading to this is as in~\cite{holroyd2009poisson} for $\R^d$, and we only give a brief sketch here: Using separability and $\sigma$-finiteness, the insertion and deletion tolerance statements of \cite[Lemma 18]{holroyd2009poisson} can be shown to hold. Now if with positive probability there are only finitely many points in $[\Pi]$ such that $\rho(x,m(x)) \ge \rho(x,o)-r$, then by deletion tolerance, we can remove these points and their matches to get an absolutely continuous point process.
Furthermore, by insertion tolerance, we can add a single point in $B(o,r)$ to get another absolutely continuous point process.
This added point will not be matched.
However, this contradicts the fact that every point in the original point process $\Pi$ (and hence also in any other absolutely continuous point process) is matched almost surely, so that almost surely every balloon eventually pops \cite{fiveballoons}.
(This follows from the ergodicity assumption and the fact that in any stable partial matching there can be at most one unmatched point.)
 
For $x \in [\Pi]$, let
  \[ T_x := \sup \{ t \ge 0 : x \in [\Pi_t]\} \]
  be the time at which the balloon at $x$ pops. Note that $2T_x$ is precisely the distance between $x$ and the center of the balloon with which it popped (see \cref{lem:balloon-matching}).
  Translated to the balloon process,
  the above gives that the set
  \[\{ x \in [\Pi] : 2T_x \ge \rho(x,o)-r \}\]
  is almost surely infinite. 
 Since every ball has finite measure, and consequently contains finitely many points almost surely, the set of times $t$ at which $R_t \le 2t+r$ is almost surely unbounded. This shows that~\eqref{eq:upper_limsup} holds.
\end{proof}

\section{Real balloons}

In this section we prove \cref{T:limsup},
and use it to establish recurrence of the Poisson balloon process in $\R^d$, i.e. \cref{thm:main}.

Let us first give a heuristic for \cref{T:limsup}, and assume $d=1$ for simplicity.
Plant at every integer $n$ a tree of height $X_n$
The ratio $X_n/n$ is the relative height of the tree at $n$ as seen from the origin (or the slope of the straight line joining the canopy and the origin).
If $\limsup X_n/n$ is positive, say at least 1, then many trees have relative height at least 1 from the origin.
Translation invariance implies that this is the case everywhere.
However, if we half the distance to a tree, it appears twice as tall, so there are many intervals, around vertices with tall trees, from which these trees have relative height at least 2.
This suggests that the $\limsup$ is at least $2$.
Repeating the same argument replacing $2$ by a large number, we obtain the result.

To formalize this, we make use of the Vitali Covering Lemma \cite{vitali1908sui} which we restate for completeness.

\begin{lemma}[Vitali Covering Lemma]\label{lem:covering}
  Let $\{B(x_i,r_i)\}_{i\in I}$ be a collection of balls in $\R^d$ of bounded radii.
  Then there exists a subset $J\subset I$ so that the balls $\{B(x_i,r_i)\}_{i\in J}$ are pairwise disjoint and their 5-fold blowups cover the original balls:
  \[
    \bigcup_{i\in I} B(x_i,r_i) \subset 
    \bigcup_{i\in J} B(x_i,5r_i).
  \]
\end{lemma}

The lemma also holds with $5$ replaced by $3+\eps$ (or even 3 for finite collections).

\begin{proof}[Proof of \cref{T:limsup}]
  Without loss of generality, assume $X$ is ergodic.
  (Otherwise we simply restrict to each ergodic component.)
  By ergodicity, $L := \limsup_{\|\bn\|\to\infty} X_\bn/\|\bn\|$ is an almost sure constant.
  Since $\limsup X_\bn$ is the supremum of the support of $X_{\mathbf{0}}$, we have $L\geq 0$.
  So from now on, we assume $L>0$ and need to show $L=\infty$.
  
  For an interval $[a,b]\subset\R_+$ and $\rho>0$, let $E_{a,b}(\rho)$ be the event
  \[
    E_{a,b}(\rho) := \left\{
      \exists \bn \text{ with } a \leq X_\bn \leq b \text{ and } X_\bn > \rho\|\bn\| \right\}.
  \]
  More generally, we denote by $E^\bu_{a,b}(\rho,\bu)$ the event $E_{a,b}(\rho)$ applied to $X$ shifted by $\bu$, i.e.
  \[
    E_{a,b}(\rho,\bu) := \left\{
      \exists \bn \text{ with } a \leq X_\bn \leq b \text{ and } X_\bn > \rho\|\bn-\bu\| \right\}.
  \]
  
  Consider now the set $S = S_{a,b}(\rho)$ of all $\bu\in\Z^d$ for which $E_{a,b}(\rho,\bu)$ holds.
  (These are $\bu$ that are close to some $\bn$ with $X_\bn \in[a,b]$.)
  By ergodicity, the set $S$ has density $\P(E_{a,b}(\rho))$. 
  From the definitions, $S$ is a union of balls
  \[
    S = \bigcup_{\bn : a \leq X_\bn \leq b}
    B\left(\bn, \tfrac{X_\bn}{\rho}\right).
  \]
  We apply the Vitali Covering Lemma to this collection of balls (whose radii are bounded by $b/\rho$), to deduce that
  there exists a subset $K\subset \{ \bn :  a \leq X_\bn \leq b\}$,
  so that the balls $\{B(\bn, X_\bn/\rho) : \bn\in K\}$ are pairwise disjoint,
  and such that
  \[ S \subset \bigcup_{\bn\in K} B\left(\bn, \tfrac{5X_\bn}{\rho}\right). \]

  For $\rho' > \rho$, consider now the set
  \[ S' = \bigcup_{\bn\in K} B\left(\bn, \tfrac{X_\bn}{\rho'}\right). \]
  These balls are pairwise disjoint, and if their radii are increased by a factor of $5\rho'/\rho$, they cover all of $S$. 
  Consequently, the density of $S'$ is at least $(\rho/5\rho')^d$ times the density of $S$, and in particular is positive.
  In other words,
  \[ \P(E_{a,b}(\rho')) \geq \left(\tfrac{\rho}{5\rho'}\right)^d \cdot \P(E_{a,b}(\rho)). \]
  The events $E_{a,b}$ are increasing in $b$, we can let $b$ tend to infinity and find 
  \[ \P(E_{a,\infty}(\rho')) \geq \left(\tfrac{\rho}{5\rho'}\right)^d \cdot \P(E_{a,\infty}(\rho)). \]

  By the definition of $L$, for every $\rho<L$ and $a$ we have $\P(E_{a,\infty}(\rho)) = 1$.
  It follows that $\P(E_{a,\infty}(\rho')) \geq (L/5\rho')^d$ for every $a$.
  Translating into the process $(X_\bn)$, and using the fact that $E_{a,\infty}$ decreasing in $a$, this gives
  \[
    \P\big( L \ge \rho'\big) \ge \P\left( \limsup_{a \to \infty} E_{a,\infty}(\rho') \right) = \lim_{a \to \infty} \P\big( E_{a,\infty}(\rho') \big)
    \geq \left(\frac{L}{5\rho'}\right)^d > 0.
  \]
  However, $L$ is constant, so for any $\rho'$ we have $L\geq \rho'$ a.s.
  Thus $L$ must be infinite.
\end{proof}

\begin{proof}[Proof of \cref{thm:main}]
For $x \in [\Pi]$, let $$T_x := \sup \{ t \ge 0 : x \in [\Pi_t]\}$$ be the time at which the balloon at $x$ pops.
Consider the process $(X_\bn)_{\bn \in \Z^d}$ defined by
\[
X_\bn := \max \{T_x: x \in [\Pi] \cap (\bn+[0,1]^d)\},
\]
and $X_\bn=0$ if there are no points of $\Pi$ in $\bn+[0,1]^d$.
This is a translation-invariant process so that \cref{T:limsup} gives that $\limsup X_\bn/\|\bn\| \in \{0,\infty\}$ almost surely.
Observe that
\[ \limsup_{t \to \infty} \frac t{R_t} = \limsup_{\|\bn\|\to\infty} \frac{X_\bn}{\|\bn\|}.\]
Indeed, if a balloon centered near $\bn$ grows to radius $X_\bn$, then at time $t=X_\bn-\eps$ we have $t/R_t \geq X_\bn/\|\bn\| + \eps'$.
Thus, $\limsup_{t \to \infty} t/R_t \in \{0,\infty\}$ almost surely.
Recalling~\eqref{eq:upper_limsup}, which states $\liminf R_t/t \leq 2$,
we conclude that $\liminf_{t \to \infty} R_t/t = 0$ almost surely, and hence the Poisson balloon process in $\R^d$ is recurrent.
\end{proof}

\begin{remark}
  Our proof of \cref{thm:main} is soft and generalizes in several directions.
  The underlying space $\R^d$ can be generalized to any reasonable metric measure space with a volume doubling property and sufficient symmetry so that one can define an ergodic action (this is needed in order to extend \cref{T:limsup}).
  The Poisson process can also be replaced by any ergodic point process that is insertion and deletion tolerant, which is almost surely non-equidistant and has no descending chains. 

  For example, \cref{thm:main} holds for the balloon process started from the point process obtained by doing a Bernoulli site percolation (with some success probability $p \in (0,1)$) on the lattice points $\Z^d \subset \R^d$ and then perturbing the vertices in space.
  While this point process does not have the full form of insertion tolerance, it does have deletion tolerance, and the proof of \cref{lem:liminf-upper-bound} can be made to work if we can insert a point near the origin after deleting some points.

  Curiously, we do not know how to establish recurrence of the balloon process on a perturbation of $\Z^d$ without percolation (i.e. for $p=1$); see \cref{prob:lattice}.
\end{remark}

\section{Well-separated sets in regular trees as factors of i.i.d.}
\label{sec:indep-sets-in-tree}

A key idea in the analysis of the Poisson balloon process in regular trees and in the hyperbolic plane is the study of well-separated point processes which are factors of i.i.d.\ processes.
In this section we give some new bounds on the density of such processes.

Let $G = (V,E)$ be a graph.
A point process on $G$ corresponds naturally to its indicator function $x\in\{0,1\}^V$.
A configuration $x \in \{0,1\}^V$ is called \textbf{$t$-separated} if no two points of the process are within distance $t$ of each other.
Equivalently, if $x_ux_v=0$ unless $d(u,v)>t$ with $d(\cdot,\cdot)$ denoting graph distance.
Note that $1$-separated configurations can be identified with independent sets.

Let $X = \{X_v\}_{v\in V}$ be a process and let $\Gamma \subset \text{Aut}(G)$ be a subgroup of the automorphism group of $G$.
For $\gamma\in \Gamma$, let $\gamma X$ denote the shifted process $\{X_{\gamma^{-1}(v)}\}_{v \in V}$.
Recall that $X$ is a \textbf{$\Gamma$-factor of an i.i.d.} process if there exists a collection of i.i.d.\ variables $Y = \{Y_v\}_{v\in V}$ and a measurable function $\varphi$ such that $\varphi(Y) =X$ and $\varphi$ is $\Gamma$-equivariant, i.e., for every $\gamma \in \Gamma$, $\gamma \varphi(Y) = \varphi(\gamma Y)$ almost surely.
Note that if $\Gamma$ acts transitively on $G$, then the distribution of $X_v$ does not depend on $v$.

Let $d \ge 3$ and consider the infinite $d$-regular tree $\T_d$. We consider random point processes $X$ on $\T_d$ which are factors of i.i.d.\ processes and which are almost surely $t$-separated.
The factors we consider are not necessarily equivariant with respect to all automorphisms of $\T_d$, but perhaps rather only with respect to a given subgroup $\Gamma$ of $\text{Aut}(\T_d)$ which acts transitively on $\T_d$.
For a transitive subgroup $\Gamma$, define
\[
  \alpha_{d,t} (\Gamma) :=
  \sup \big\{ \P(X_v=1) :  \text{$(X_v)_{v\in\T_d}$ is $t$-separated and a $\Gamma$-equivariant factor of i.i.d.} \big \} .
\]
Define also
\[ \alpha_{d,t} :=
  \sup\big\{\alpha_{d,t}(\Gamma):\Gamma \text{ is a transitive subgroup of Aut$(\T_d)$} \big\}. \]

Consider an automorphism-invariant $t$-separated process $(X_v)_{v\in\T_d}$.
Each $v$ with $X_v=1$ is the center of a ball of radius $t/2$, and these balls are disjoint. 
It follows by an easy mass transport argument that for such a process $\P(X_v=1) \leq |B(x,t/2)|^{-1} \sim C(d-1)^{-t/2}$.
(With a slight variation depending on the parity of $t$.)
The next result gives an improvement on this simple bound.

\begin{thm}\label{thm:tree-independent-set-bound}
  For any $d \ge 3$ and $t \ge 1$, we have 
  \[ \alpha_{d,t} \le \frac{2t \log (d-1)}{(d-1)^{t} }\]
\end{thm}

This type of bound goes back to Bollob{\'a}s~\cite{bollobas1981independence}, who proved an analogous upper bound for the density of the largest independent set (i.e., for $t=1$) in the $d$-regular configuration model.
That result can easily be translated to the infinite $d$-regular tree to yield 
\begin{equation}\label{eq:bollobas}
  \alpha_{d,1} (\text{Aut}(\T_d)) \le  \frac{2\log d}{d} .
\end{equation}
Indeed, for this one exploits the well-known facts that the $d$-regular configuration model locally converges to the $d$-regular tree and that a factor of i.i.d.\ can be approximated by a block factor (which gives a local algorithm).
A sharpening of the upper bound~\eqref{eq:bollobas} for large $d$ was obtained by Rahman and Virag~\cite{rahman2017local} who removed the factor $2$ asymptotically as $d\to\infty$.
A corresponding lower bound follows from a construction due to Lauer and Wormald~\cite{lauer2007large}, so that
\[ \alpha_{d,1} (\text{Aut}(\T_d)) = \frac{\log d}{ d}(1+o(1)). \]

\cref{thm:tree-independent-set-bound} extends~\eqref{eq:bollobas} in two directions:
First, it applies to separation distance $t$ larger than~1.
Secondly, it relaxes the equivariance requirement from the full automorphism group $\text{Aut}(\T_d)$ to any transitive subgroup $\Gamma$.
Our strategy of proof is based on a similar approach to the one in~\cite{bollobas1981independence} leading to~\eqref{eq:bollobas}.
The extension to $t>1$ requires a more complex computation, but is otherwise rather straightforward.
Allowing for an arbitrary transitive group $\Gamma$ takes a more substantial modification.
We consider the following variation of the configuration model.
Start with $n$ vertices, where $n$ is assumed to be even.
On these, take a disjoint union of $d$ independent perfect matchings.
We color the edges by assigning the edges of the $i$th matching color $i$.
This yields a random (possibly disconnected) $d$-regular graph $G_{n,d}^*$ on $n$ vertices (with no self-loops, but possibly with multiple edges; indeed, the number of double edges is roughly Poisson with mean $d(d-1)/2$ as $n\to\infty$), in which the edges are colored with colors in $\{1,\dots,d\}$ and the $d$ edges incident to any vertex have distinct colors.
Let $\T_d^*$ be the edge-colored version of $\T_d$, in which the $d$ edges incident to any vertex have distinct colors in $\{1,\dots,d\}$.

\begin{prop}\label{prop:local-convergence}
  Fix $d \ge 3$.
  Then $G_{n,d}^*$ converges locally to $\T_d^*$ as $n \to \infty$.
\end{prop}

The above convergence holds in local, marked topology, which roughly means that for any $r \ge 1$ and most vertices $v$, the ball of radius $r$ around $v$ isomorphic to the ball of radius $r$ in $\T_d^*$ (viewed up to a rooted isomorphism preserving the edge colors).
See e.g. \cite{aldouslyons} for a precise definition.

\begin{proof}[Proof of \cref{prop:local-convergence}]
  This is standard, so we only provide a brief sketch.
  We can explore the ball of radius $r$ around a vertex $v$ in $G_{n,d}^*$ by finding the match of all the half edges incident to $v$, then match of all it's neighbours and so on, and continue this for $r$ steps.
  If at some step of this exploration process we have already seen $k$ vertices,
  then as we explore an edge of one of the matchings, 
  we connect to a new vertex with probability at least $1-k/n$.
  We conclude that for $d,r$ fixed and as $n\to\infty$, 
  with high probability a new vertex is discovered in all such steps until the ball of radius $r$ is explored.
  This is what is needed to ensure that the ball is a tree.
\end{proof}

\begin{prop}\label{prop:indep-set-bound}
  Fix $d \ge 3$ and $t \ge 1$.
  Then the largest $t$-separated set in $G_{n,d}^*$ has size at most
  \[ \frac{2 t \log(d-1)}{(d-1)^{t}} \cdot n + o(n) \]
  with probability converging to 1 as $n \to \infty$.
\end{prop}

Before proving \cref{prop:indep-set-bound}, let us show how these two propositions imply the theorem.

\begin{proof}[Proof of \cref{thm:tree-independent-set-bound}]
Let $\Gamma$ be a transitive group of automorphisms of $\T_d$, and let $X$ be a $t$-separated set in $\T_d$ which is a $\Gamma$-factor of an i.i.d.\ process $Y$.
For fixed $r \ge 0$, denote
\[ Z_v = \E[X_v\mid (Y_u)_{u \in B(v,r)}]. \]
By Levy's 0-1 law, since $X_v\in\{0,1\}$ is a factor of $Y$ we have $Z_v \to X_v$ as $r \to \infty$ almost surely.
Define a process $X'$ by
\[ X'_v := \mathbf{1}_{\{ Z_v > 0.5 \text{ and } Z_w \le 0.5 \text{ for all }w \in B(v,t) \setminus \{v\}\}}  .\]
Then $X'$ is a block factor of $Y$ with radius $r+t$ (meaning that there is a $\Gamma$-equivariant map $\varphi$ such that $\varphi(Y) = X'$ and $\varphi(Y)$ at a vertex $v$ depends only on the value of $Y$ in a ball of radius $r+t$ around $v$ almost surely for any $v$).
Moreover $X'$ is $t$-separated, and $X'_v \to X_v$ almost surely as $r \to \infty$.
In particular,  $\P(X'_v=1) \to \P(X_v=1)$ as $r \to \infty$.
Thus it suffices to prove the stated bound for $\P(X'_v=1)$.

Fix a vertex $\rho \in \T_d^*$, let $\cB$ be the ball of radius $r'=r+t$ around $\rho$ in $\T^*_d$ (seen as an edge colored graph, rooted at $\rho$), and let $B$ be its vertex set.
Suppose that $Y$ takes values in $S$ and let $\varphi \colon S^B \to \{0,1\}$ be measurable and such that $X'_\rho = \varphi(Y|_B)$ almost surely.
In fact, we may also require that $\varphi$ has some form of invariance related to $\Gamma$, but we will not need this. Consider the random graph $G^*_{n,d}$ and let $\cB^*_u$ be the rooted colored subgraph induced by the ball of radius $r'$ around $u$ (and rooted at $u$). Let $B^*_u$ be the vertex set of $\cB^*_u$. Let $Y^*$ be an independent random field on the vertices of $G^*_{n,d}$, consisting of independent copies of $Y_\rho$. Define another random field $X^*$ on the vertices of $G^*_{n,d}$ by
\[ X^*_u = \begin{cases} \varphi(Y^*|_{B^*_u}) &\text{if $\cB^*_u$ and $\cB$ are isomorphic as rooted colored graphs} \\0 &\text{otherwise} \end{cases} .\]
Note that when $\cB^*_u$ and $\cB$ are isomorphic as rooted colored graphs, there is a unique isomorphism (since each ``internal'' vertex has exactly one incident edge of each of the $d$ colors), and so there is a unique way to interpret $Y^*|_{B^*_u}$ as an element of $S^B$. This ensures that $X^*$ is well defined.
Observe that $X^*$ is almost surely $t$-separated and $\P(X^*_u=1)$ does not depend on $u$. Thus, \cref{prop:indep-set-bound} implies that
\[ \P(X^*_u=1) \le \frac{2t \log(d-1)}{(d-1)^{t}} + o(1) , \qquad \text{ as } n \to \infty. \]
Since \cref{prop:local-convergence} implies that $\cB^*_u$ and $\cB$ are isomorphic as rooted colored graphs with high probability as $n\to\infty$, we conclude that
\[ \P(X'_\rho=1) \le \frac{2t\log(d-1)}{(d-1)^{t-1}}. \]
This completes the proof. 
\end{proof}

\subsection{Proof of \cref{prop:indep-set-bound}}

Our proof follows the fundamental method used by Bollob{\'a}s~\cite{bollobas1981independence} for the case $t=1$:
We bound the probability that a fixed set of size $\alpha n$ is $t$-separated in $G^*_{n,d}$, and then conclude by a union bound that $t$-separated sets of this size typically do not exist for
\[ \alpha := \frac{2t\log(d-1)}{(d-1)^{t}}.
\]

Denote the number of perfect matching of $n$ elements for even $n$ (i.e., partitions of $n$ elements into unordered pairs) by
\[ N = N(n) := \frac{n!}{2^{n/2} (n/2)!} .\]
Thus $G^*_{n,d}$ is uniformly chosen among $N^d$ possible edge colored graphs.
Roughly speaking, we estimate the probability that a given set $I$ of size $k=\alpha n$ is $t$-separated and then union bound over $I$.
We consider separately the cases when the separation distance $t$ is even or odd.
Of course, one can obtain a weaker bound for the case $t=2s+1$ from the case $t=2s$.
Nevertheless, we provide a proof of the tighter bound in both cases as it maybe of independent interest.

Before starting, let us introduce some notation.
As noted, we will be looking at sets $I$ of size $\alpha n$.
The volume of a ball of radius $i$ in the $d$-regular tree $\T_d$ and the number of vertices on its (internal) boundary are denoted by
\begin{align*}
  b_i & := \frac{d(d-1)^i-2}{d-2}, &
  \partial b_i & := b_i - b_{i-1} = d(d-1)^{i-1}.
\end{align*}
We also define
\begin{align*}
  \alpha_i &:= \alpha b_i = \alpha \frac{d(d-1)^i-2}{d-2}, &
  \beta_i & := \alpha \partial b_{i+1} = \alpha d(d-1)^i, &
  \gamma &:= 2\alpha \frac{(d-1)^s-1}{d-2}.
\end{align*}

\paragraph{Even separation distance.}
Fix an even $t=2s$.
Let $B_v$ denote the ball of radius $s$ around $v$ in $G^*_{n,d}$,
and let $\cE(I)$ be the event that $I$ is $t$-separated,
and that for each $v\in I$ the ball of radius $s$ around $v$ in $G^*_{n,d}$ is a tree.

\begin{claim}
  With the above notation,
  \[
    \P(\cE(I)) \asymp \Big[(1-\alpha)^{(1-\alpha)} (1-\alpha_s)^{-(1-\alpha_s)} (1-\gamma)^{d(1-\gamma)/2}\Big]^n,
  \]
  where the constants implicit in $\asymp$ depend only on $\alpha$ and $d$.
\end{claim}

\begin{proof}
  Note that $I$ is $t$-separated exactly when the balls $(B_v)_{v \in I}$ are pairwise disjoint.
  Let $I_k$ be the set of vertices at distance $k$ from $I$, so that $I_0=I$.
  Observe that $\cE(I)$ occurs exactly when all of the following hold:
  \begin{itemize}
  \item Each of the $kd$ half-edges incident to $I_0=I$ are matched to half-edges belonging to $kd$ distinct vertices (which form $I_1$), disjoint from $I$.
  \item Inductively, for every $1 \le i \le s-1$, each of the $kd(d-1)^i$ half-edges incident to $I_i$ but not to $I_{i-1}$ are matched to half-edges belonging to $kd(d-1)^i$ distinct vertices $I_{i+1}$, disjoint from $I_0 \cup \dots \cup I_i$.
  \end{itemize}
  It follows (writing $(m)_j$ for $\frac{m!}{(m-j)!} = m(m-1)\cdots(m-j+1)$) that \[ \P(\cE(I)) = \prod_{i=0}^{s-1} \left(n-k \tfrac{d(d-1)^i-2}{d-2}\right)_{kd(d-1)^i} \cdot \frac{\left[N\left(n - 2k\frac{(d-1)^s-1}{d-2}\right)\right]^d}{N(n)^d} .\]
  The above expression can be easily proved by induction. To explain the product term,
  suppose we have already obtained the ball of radius $i$ in all of the $k$ components for $i \ge 0$.
  To build the $(i+1)$th layer, we need to match $k\partial b_{i+1}$ many half-edges with a color matching half-edge from $n-kb_{i}$ many vertices, and the matching has to be such that each of the remaining vertices is matched to exactly one of the half-edges. The number of ways to do this is $(n-kb_{i})_{k\partial b_i}$. We now take a product over these terms from $i=0$ to $s-1$. 

  The remaining term in the numerator counts all the possible ways to obtain the remaining matching. Indeed, it is easy to see by induction on $i$ that for each $1 \le m \le d$ and $v \in I$, the number of edges of color $m$ joining a vertex at distance $i$ to a vertex at distance $i+1$ from $v$ is the same for every $m$ and $v$, namely, $$\frac1d \partial b_{i+1} = (d-1)^{i}.$$ Therefore the number of half-edges of color $m$ used up is$$2k\sum_{i=0}^{s-1} (d-1)^i = 2k\frac{(d-1)^s-1}{d-2}.$$
  Counting the number of matchings of all the remaining half-edges gives the desired expression.

  \medskip

  In terms of $\alpha_i,\beta,\gamma$ defined above, this becomes
  \[ \P(\cE(I)) = \prod_{i=0}^{s-1} (n(1-\alpha_i))_{\beta_i n} \cdot \frac{N(n(1-\gamma))^d}{N(n)^d} .\]
  Using Sterling's formula
  we get that
  \[ (n(1-\alpha_i))_{\beta_i n} \asymp (\tfrac ne)^{\beta_i n} (1-\alpha_i)^{n(1-\alpha_i)} (1-\alpha_i-\beta_i)^{-n(1-\alpha_i-\beta_i)} \]
  and
  \[ \frac{N(n(1-\gamma))}{N(n)} \asymp (\tfrac en)^{n \gamma/2} (1-\gamma)^{n(1-\gamma)/2} .\]
  Noting that $\beta_i=\alpha_{i+1}-\alpha_i$ and $\alpha_s-\alpha_0=\frac d2 \gamma$, we get the claimed asymptotics
  \[ \P(\cE(I)) \asymp \Big[(1-\alpha)^{(1-\alpha)} (1-\alpha_s)^{-(1-\alpha_s)} (1-\gamma)^{d(1-\gamma)/2}\Big]^n.
    \qedhere
  \]
\end{proof}

\begin{proof}[Proof of \cref{prop:indep-set-bound} for $t$ even]
  Denote
  \[
    p(\alpha) := -\log [(1-\alpha)^{(1-\alpha)} (1-\alpha_s)^{-(1-\alpha_s)} (1-\gamma)^{d(1-\gamma)/2}\Big]
  \]
  so that $\P(\cE(I)) \asymp e^{-n p(\alpha)}$.
  The number of sets $I$ to consider is
  \[
    \binom{n}{\alpha n} \leq e^{H(\alpha) n},
  \]
  where $H(\alpha) := -\alpha \log (\alpha) -(1-\alpha) \log (1-\alpha)$.
  (To see the inequality, note that $e^{-H(\alpha) n} \binom{n}{\alpha n}$ is one of the terms in the binomial expansion of $1=(\alpha+(1-\alpha))^n$.)
  By a union bound over the choices for $I$ of size $\alpha n$, the probability that $\cE(I)$ occurs for some such $I$ is at most $C e^{(H(\alpha)-p(\alpha))n}$.
  By Markov's inequality (recalling also that $E_v$ is the event that the ball around $v$ is a tree),
  \[
    \P(\text{there exists a $t$-separated set of size $\alpha n + m$}) \le \frac{\P(E_v^c) n}{m} +e^{(H(\alpha) - p(\alpha))n}.
  \]
  Indeed, if there exists a $t$-separated set of size $\alpha n+m$, then either it contains more than $m$ vertices $v$ for which $E_v$ does not occur, or else it contains a subset $I$ of size $\alpha n$ for which $\cE(I)$ occurs.
  It remains to check that $p(\alpha) < H(\alpha)$.
  Given that, set $m = \sqrt{\P(E_v^c)} n = o(n)$.
  Since $E_v$ occurs with probability tending to 1 as $n \to \infty$, it follows that the largest $t$-separated set has size at most $\alpha n + o(n)$ with probability tending to 1.

  \medskip

  Thus it remains to show that $p(\alpha)>H(\alpha)$, which is nothing more than a tricky algebra problem.
  This is equivalent to 
  \[ \tfrac d2 (1-\gamma) \log (1-\gamma) < (1-\alpha_s)\log (1-\alpha_s) + \alpha \log \alpha .\]
  We have the Taylor series $(1-x)\log(1-x) = -x + \sum_{i=2}^\infty x^i/(i^2-i)$.
  Using this and recalling that $\alpha_s - \frac d2 \gamma = \alpha$, we arrive at the equivalent inequality
  \[
    \alpha(1-\log\alpha) < \sum_{i=2}^\infty \frac{\alpha_s^i - \frac{d}2 \gamma^i}{i^2-i}. 
  \]
  Using that $\alpha_s \ge \frac d2 \gamma$, it easily follows that each term on the right-hand side is non-negative.
  Thus, discarding all terms other than $i=2$, we see that it suffices to show that
  \[ \alpha_s^2 - \tfrac d2 \gamma^2 > 2\alpha(1-\log\alpha) .\]
Plugging in the definitions of $\alpha_s$ and $\gamma$, and simplifying, we get the equivalent inequality
\[ \frac{d(d-1)^{2s} - 2}{d-2} > \frac 2\alpha(1-\log\alpha) .\]
Noting that the left hand side is at least $(d-1)^{2s}$ (since $\frac{dx-2}{d-2} \ge x$ if and only if $x \ge 1$), we get the sufficient inequality
\[ (d-1)^{2s} \ge \frac 2\alpha(1-\log\alpha). \]
This is indeed satisfied by (and explains) our choice of $\alpha =\frac{4s \log (d-1)}{(d-1)^{2s}}$.
\end{proof}

\paragraph{Odd separation distance.}
The odd case is very similar to the even case, with a few minor differences that we shall emphasize.
Suppose that $t=2s-1$ is odd.
As before, let $B_v$ denote the ball of radius $s$ around $v$ in $G^*_{n,d}$.
Then $I$ is $t$-separated exactly when the \emph{edge sets} of $(B_v)_{v \in I}$ are pairwise disjoint, but in contrast to the even case it is allowed for boundary vertices to overlap.
Let $E_v$ be the event that $B_v$ is a tree, and let $\cE(I)$ be the event that $I$ is $t$-separated and $E_v$ occurs for all $v\in I$.

\begin{claim}
  With the above notations, $\P(\cE(I)) \asymp e^{-p(\alpha)n}$, where
  \[ e^{-p(\alpha)} := (1-\alpha)^{(1-\alpha)} (1-\alpha_{s-1})^{(d-1)(1-\alpha_{s-1})} (1-\gamma)^{-d(1-\gamma)/2} .\]
\end{claim}

\begin{proof}
Observe that $\cE(I)$ occurs exactly when all of the following hold:
\begin{itemize}
\item If $s>1$, then for every $0 \le i \le s-2$, each of the $kd(d-1)^i$ half-edges incident to $I_i$ but not to $I_0 \cup \dots \cup I_{i-1}$ are matched to half-edges belonging to $kd(d-1)^i$ distinct vertices $I_{i+1}$, disjoint from $I_0 \cup \dots \cup I_i$ (where $I_0:=I$).
\item Each of the $kd(d-1)^{s-1}$ half-edges incident to $I_{s-1}$ but not to $I_0 \cup \dots \cup I_{s-2}$ are matched to half-edges belonging to some set of vertices $I_s$, disjoint from $I_0 \cup \dots \cup I_{s-1}$.
  Here, $I_s$ might have size smaller than $kd(d-1)^{s-1}$, as any vertex in $I_s$ can be used several times, namely, once for each color.
\end{itemize}
It follows (with $\alpha_i$, $\beta_i$ and $\gamma$ defined as before) that
\[ \P(\cE(I)) = \prod_{i=0}^{s-2} (n(1-\alpha_i))_{\beta_i n} \cdot \Big((n(1-\alpha_{s-1}))_{\beta_{s-1}n/d}\Big)^d \cdot \frac{N(n(1-\gamma))^d}{N(n)^d} .\]
Using Sterling's formula and that $\alpha_{s-1}+\frac1d \beta_{s-1}=\gamma$, we get the claimed probability.
\end{proof}

\begin{proof}[Proof of \cref{prop:indep-set-bound} for $t$ odd]
  As before, it suffices to show that $p(\alpha)>H(\alpha)$.
  This is equivalent to
  \[ (d-1)(1-\alpha_{s-1})\log (1-\alpha_{s-1}) < \tfrac d2 (1-\gamma)\log(1-\gamma) + \alpha \log\alpha .\]
  Using the same Taylor expansion and using that $\frac d2 \gamma - (d-1)\alpha_{s-1} = \alpha$, we arrive at the equivalent inequality
  \[ \sum_{i=2}^\infty \frac{ \tfrac d2 \gamma^i - (d-1)\alpha_{s-1}^i}{i^2-i} > \alpha(1- \log\alpha) .\]
  Noting that all terms in the sum are non-negative, it suffices that
  \[ \tfrac d2 \gamma^2 - (d-1)\alpha_{s-1}^2 > 2\alpha(1-\log\alpha) .\]
  Plugging in the definitions of $\alpha_{s-1}$ and $\gamma$, we get
  \[ \frac{d(d-1)^{2s-1} - 2}{d-2} > \frac 2\alpha(1+\log\tfrac1\alpha) ,\]
  which holds for $\alpha=\frac{2(2s-1) \log (d-1)}{(d-1)^{2s-1}}$.
\end{proof}

\section{Balloons on trees}\label{sec:balloons-in-tree}

In this section, we prove \cref{thm:main3}.
Recall from~\eqref{eq:upper_limsup} that $\liminf_{t \to \infty} R_t/t \le 2$ almost surely.
Our goal is thus to show the opposite inequality.
We in fact show that, almost surely,
\[
R_t \ge 2t - O(\log t) \qquad\text{as }t \to \infty.
\]

We identify the discrete tree $\T_d$ with the vertices of the continuous space $\cT_d$.
Given a collection of points $P$ in $\cT_d$, let $\pi(P)$ denote their projection to $\T_d$.
Specifically, $x \in \pi(P)$ if and only if $P$ contains a point at distance less than $1/2$ from $x$.
(For concreteness we discard points lying precisely in the middle of an edge, though this is not very important.)
Observe that for any $t \ge 1/2$, $\pi([\Pi_t])$ is almost surely $(2t-1)$-separated in $\T_d$.
Since the Poisson process is invariant under $\text{Aut}(\T_d)$, $\pi([\Pi_t])$ is an Aut$(\T_d)$-equivariant factor of an i.i.d.\ process on $\T_d$.
Indeed, on each vertex of $\T_d$, we attach $d$ (ordered collection) of i.i.d.\ Poisson processes on the interval $[0,1/2)$ and a Uniform$[0,1]$ random variable independent of everything else.
The latter creates an ordering of the half-edges incident to a vertex, to which we assign the $d$ Poisson processes in order.
We can run the balloon process to describe $\pi([\Pi_t])$. Clearly, this is an equivariant function of these i.i.d.\ variables.

Therefore, \cref{thm:tree-independent-set-bound} implies that
\[ \P(\pi([\Pi_t])_v=1) \le C_d t (d-1)^{-2t} \]
(we use the obvious bound $\alpha_{d,t} (\text{Aut}(\T_d)) \le \alpha_{d,t}$ here).
Since the ball volumes grow like $(d-1)^s$ as $s \to \infty$, we have that for any positive integer $s$,
\[ \P(R_t \le s) \le C'_d t (d-1)^{s-2t} .\]
Taking $s(t) := 2t - \frac{3\log t}{\log(d-1)}$, we get that
\[ \P(R_t \le s(t)) \le C'_d t^{-2} .\]
By Borel-Cantelli, we have that, almost surely,
\[ R_n > s(n) \qquad\text{for all large enough integer }n .\]
Since $R_t$ is increasing, we may interpolate between integers.
Thus, almost surely,
\[ \limsup_{t \to \infty} \frac{2t-R_t}{\log t} \le \frac{3}{\log(d-1)} .\]

\section{Balloons in the hyperbolic plane}

In this section, we prove \cref{thm:main2}.
Throughout this section, our measured metric space is the hyperbolic plane $\HH$ with measure $\mu$ and metric $\rho$.
Recall from~\eqref{eq:upper_limsup} that $\liminf_{t \to \infty} R_t/t \le 2$ almost surely. On the other hand, the analogue of \cref{lem:covering} does not lead to an analogue of \cref{T:limsup} because of the lack of volume doubling property. Indeed, \cref{thm:main2} states that $\liminf_{t \to \infty} R_t/t$ is almost surely strictly positive, so the analogue of \cref{T:limsup} simply fails for the hyperbolic plane.

The proof of the lower bound on $\liminf_{t \to \infty} R_t/t$ is morally similar to that in \cref{sec:balloons-in-tree} for a tree.
For the continuous tree, we divided $\cT_d$ into ``fundamental domains'' which were balls of radius $1/2$ around the vertices of the tree.
Here we divide the hyperbolic plane into fundamental domains which are the ideal triangles dual to some embedding of a 3-regular tree $\T_3$ in $\HH$ (see \cref{fig:tree_in_H}).
We embed the tree in such a way that the group $\Gamma$ of isometries of $\HH$ that preserve $\T_3$ acts transitively on $\T_3$. 

\begin{figure}
  \centering
  \includegraphics[width=.55\textwidth]{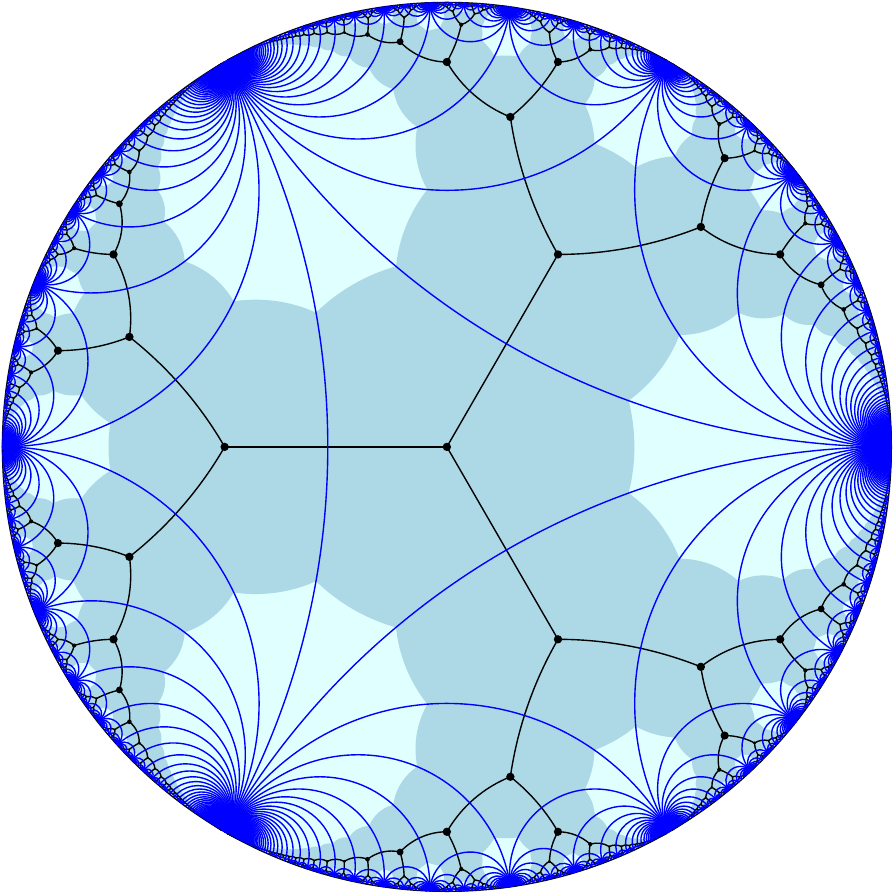}
  \caption{The embedding of the 3-regular tree in the hyperbolic plane.}
  \label{fig:tree_in_H}
\end{figure}

Let us define the tessellation more carefully.
We shall use the Poincar\'e disc model of  the hyperbolic plane (with constant curvature $-1$).
Start with an ideal triangle with corners at $1, e^{i\frac{2\pi}{3}}$ and $e^{i\frac{4\pi}{3}}$.
Iteratively reflect along the edges to obtain the tessellation.
The iterative images of the origin under this mapping will define the vertices of the embedding of the 3-regular tree.
This defines a proper embedding of the 3-regular tree in $\HH$.

A significant difference between the situation in the hyperbolic plane and a regular tree is that the fundamental domains here have infinite diameter (whereas they were bounded in the tree case).
A second difference is that $\Gamma$ is not the full automorphism group of $\T_3$.
This will not be a problem given what we have shown in~\cref{sec:indep-sets-in-tree} (we allowed the factors there to be not completely invariant).

Let $p \colon \HH \to \T_3$ be the projection that maps a point $x \in \HH$ to the a nearest point in $\T_3$ (breaking ties arbitrarily).
For $v \in \T_3$, let $\Delta_v$ be the ideal triangle dual to $v$ (i.e., $\Delta_v$ is the closure of $\pi^{-1}(v)$).
As noted, the projection $p$ can heavily distort, as arbitrarily far away points may be mapped to neighboring vertices of the tree.
We deal with this by truncating the triangles.
For some $r>0$ to be fixed shortly,
split each triangle into a disjoint union $\Delta_v = \Delta'_v \cup K_v$, where
\begin{align*}
  \Delta'_v &= \{x \in \Delta_v : d_\HH(x,v)\leq r\},
  & K_v &= \{x \in \Delta_v : d_\HH(x,v) > r\}.
\end{align*}
Thus $K_v$ consists of three caps near the cusps of $\Delta_v$ (see \cref{fig:tree_in_H}).
Since the hyperbolic area (denoted $\mu$) of each triangle is $\pi$, we fix $r$ so that $\mu(\Delta'_v) = \mu(K_v) = \pi/2$
(this choice is fairly arbitrary).
The main geometric step is that restricted to the $\Delta'_v$ the projection has bounded distortion:

\begin{lemma}\label{C:hyperbolic-tree-dist}
  With the above construction, for any $x,y\in \bigcup \Delta'_v$ we have
  \begin{equation}\label{eq:hyperbolic-tree-dist}
    \rho(x,y) \le a \cdot d_{\T_3}(\pi(x),\pi(y)) + c,
  \end{equation}
  where $a=2\log\frac{1+\sqrt5}{2}$ and
  $c = 2r+\log 3-2\log\frac{1+\sqrt5}{2}$. 
\end{lemma}

\begin{proof}
  Consider an ideal triangle in $\HH$ with center at $u$.
  Let $z,z',z''$ be the projection of $u$ onto its three edges.
  A straightforward calculation (which we omit) yields that
  \begin{align*}
    \rho(u,z) &= \frac12 \log 3, &
    \rho(z,z') &= a := 2\log\frac{1+\sqrt5}{2}.
  \end{align*}

  Suppose now that $d_{\T_3}(\pi(x),\pi(y))=\ell$,
  and let $z_1,\dots,z_\ell$ be the midpoints of the edges on the path (as embedded in $\HH$).
  We have that $\rho(x,z_1) \leq r+\rho(\pi(x),z_1) = r+\frac12\log 3$, and the same bounds holds for $\rho(y,z_\ell)$.
  Since $\rho(z_i,z_{i+1}) = a$, the triangle inequality gives
  \[
    \rho(x,y) \leq 2r + \log 3 + (\ell-1)a,
  \]
  which is the claimed bound.
\end{proof}

\begin{remark}
  The additive constant $c$ in~\eqref{eq:hyperbolic-tree-dist} is not significant later on.
  The multiplicative constant $a$ is the best possible.
  Indeed, if the path in the tree is a zig-zag path that alternates left and right turns, then the points $z_i$ are all on a hyperbolic geodesic, and thus $\rho(z_1,z_\ell) = (\ell-1)a$.
\end{remark}

\begin{proof}[Proof of \cref{thm:main2}]
  Recall that $\Pi_t$ is the counting measure for the set of centers still active at time $t$.
  Let $\lambda_t$ denote the intensity of $\Pi_t$, which, by the invariance of the balloon process to isometries of $\HH$, is defined by
  \[ \lambda_t := \frac{\E\Pi_t(A)}{\mu(A)} ,\]
  where $A$ is any subset of $\HH$ with finite positive area.
  (This is because $A \mapsto \E \Pi_t(A)$ is an invariant measure, and any such measure is a constant multiple of the hyperbolic measure.)
  Our first goal is to bound $\lambda_t = \frac2\pi \E\Pi_t(\Delta'_v)$.

  Suppose that $t$ is large, and let $\Pi'_t$ denote the restriction of $\Pi_t$ to $\bigcup_v \Delta'_v$.
  Note that, almost surely, any two points in $[\Pi'_t]$ are at distance greater than $2t$ in $\HH$ (since this is true for $[\Pi_t]$).
  For $t>2r$ it follows that $[\Pi'_t]$ contains at most one point in each truncated triangle $\Delta'_v$ (since $\Delta'_v$ has diameter $2r$).
  Let $p(\Pi'_t) \subset \T_3$ denote the projection of $[\Pi'_t]$.
  \cref{C:hyperbolic-tree-dist} implies that $p(\Pi_t)$ is almost surely $\frac{2(t-c)}a$-separated in $\T_3$, with $a$ and $c$ from the lemma.
  Observe also that $p(\Pi'_t)$ is a $\Gamma$-factor of an i.i.d.\ process on $\T_3$ (this is similar to the case of the tree in \cref{sec:balloons-in-tree}).
  Thus, \cref{thm:tree-independent-set-bound} applied for $\frac{2(t-c)}{a}$-separated processes on $\T_3$ implies that
  \[ \E \Pi_t(\Delta'_v)
    = \E p(\Pi'_t)_v
    \leq \frac{4(t-c)}{a} \log2 \cdot 2^{-2(t-c)/a}. \]
  We conclude that
  \[ \lambda_t \le \frac{8t\log 2}{\pi a} \cdot 2^{-2(t-c)/a}
    = C t 4^{-t/a} .\]

  Let $V(s) := 4\pi \sinh^2 (s/2)$ denote the area of of a ball of radius $s$ in $\HH$ (recall that the curvature is taken to be $-1$).
  By Markov's inequality, using $V(s)\leq \pi e^s$, we see that
  \begin{align*}
    \P(R_t \le s) &= \P(\Pi_t(V(s)) \ge 1) \\
    &\leq V(s) \lambda_t \\
    &\leq C t e^s 4^{-t/a}.   \end{align*}
  This is summable for $t\in\N$ if we take
  \[ s=s(t) = \frac{\log 4}{a} t - 3\log t. \]
  Applying Borel--Cantelli, and interpolating using monotonicity for non-integer $t$ (as before), we conclude that
  \[ \liminf_{t \to \infty} \frac{R_t}t \ge \frac{\log 4}a \qquad\text{almost surely}. \]
  Since $a=2\log\frac{1+\sqrt5}{2}$, we are done.
\end{proof}

\section{Open problems}

\begin{problem}
  Give quantitative information about $R_t/t$ as $t\to\infty$ for the Poisson balloon process on the Euclidean space $\R^d$:
  \begin{itemize}[nosep]
  \item Is $\limsup_{t \to \infty} \frac{R_t}t$ finite?
  \item Does $\frac{R_t}t$ converge in distribution as $t \to \infty$?
  \item Does the set $A_c := \{ \log t : R_t \le ct \} \cap [0,\infty)$ have a positive density (for fixed $c>0$)?
  \end{itemize}
\end{problem}

\begin{problem}
  Does the rescaled point process $\{ x : xt \in [\Pi_t] \}$ converge in distribution as $t \to \infty$?
\end{problem}

\begin{problem}
  Give quantitative information about $R_t/t$ as $t\to\infty$ for the Poisson balloon process in the hyperbolic plane $\HH$:
  \begin{itemize}[nosep]
  \item What is $\liminf_{t \to \infty} \frac{R_t}{t}$? Does it depend on the curvature or intensity of the Poisson process? Is it always 2?
  \item What is $\limsup_{t \to \infty} \frac{R_t}t$? Does it depend on the curvature/intensity? Is it always finite?
  \end{itemize}
\end{problem}

We mention that in any reasonable space (all one needs is that every balloon eventually pops), $\limsup_{t \to \infty} \frac{R_t}t \ge 1$ almost surely. In $\R^d$, using our result on the liminf in \cref{thm:main}, it is not hard to see that $\limsup_{t \to \infty} \frac{R_t}t \ge 2$ almost surely. Understanding $\limsup_{t \to \infty} \frac{R_t}t$ on the regular tree $\cT_d$ is also of interest.

\begin{problem}
  Prove that the Poisson balloon process on the $d$-dimensional hyperbolic space $\HH^d$ is transient.
\end{problem}

\begin{problem}\label{prob:lattice}
  Consider the point process in $\R^d$ obtained by independently perturbing (in some reasonable way) the points of the lattice $\Z^d$.
  Show that the corresponding balloon process is recurrent.
\end{problem}

\begin{problem}
  Consider a modified Poisson balloon process in $\R^d$ or $\HH^d$ where each balloon has an independent random rate of growth chosen according to some distribution $\nu$ on $(0,\infty)$. When is this balloon process recurrent/transient (assuming is it well defined)?
\end{problem}

In $\R^d$, when $\nu$ is supported in $[a,b]$ for some $0<a<b<\infty$, simple modifications of the arguments yield that this balloon process is well defined, that it is recurrent and that $\liminf_{t \to \infty} \frac{R_t}t=0$ almost surely.

\bibliography{balloons}

\end{document}